\documentclass[12pt,a4paper]{article}
\usepackage{amsmath,amssymb,amsthm,graphicx,color,bbm}
\usepackage[left=1in,right=1in,top=1in,bottom=1in]{geometry}
\usepackage{cite}
\usepackage[british]{babel}
\usepackage{hyperref} 
\usepackage{enumitem}

\hypersetup{
    colorlinks=false,
    pdfborder={0 0 0},
}

\newtheorem{theorem}{Theorem}[section]
\newtheorem{corollary}[theorem]{Corollary}
\newtheorem{proposition}[theorem]{Proposition}
\newtheorem{lemma}[theorem]{Lemma}
 \newtheorem{conjecture}[theorem]{Conjecture}

\numberwithin{equation}{section}

\theoremstyle{definition}
\newtheorem{definition}[theorem]{Definition}

\newtheorem{problem}[theorem]{Problem}

\theoremstyle{remark}

\newtheorem*{remark*}{Remark}

 \allowdisplaybreaks

\newcommand{\1}[1]{{\mathbbm{1}\mkern -1.5mu}{\{#1\}}}
\newcommand{\2}[1]{{\mathbbm{1}}_{#1}}

\newcommand{\R}{{\mathbb R}}
\newcommand{\Z}{{\mathbb Z}}
\newcommand{\N}{{\mathbb N}}
\newcommand{\ZP}{{\mathbb Z}_+}
\newcommand{\RP}{{\mathbb R}_+}
\newcommand{\Sp}[1]{{\mathbb S}^{#1}}

\DeclareMathOperator{\Exp}{\mathbb{E}}
\renewcommand{\Pr}{{\mathbb P}}

\DeclareMathOperator{\trace}{tr}

\newcommand{\eps}{\varepsilon}

\newcommand{\re}{{\mathrm{e}}}
\newcommand{\rc}{{\mathrm{c}}}

\newcommand{\cC}{{\mathcal C}}

\newcommand{\cF}{{\mathcal F}}
\newcommand{\cG}{{\mathcal G}}

\newcommand{\cS}{{\mathcal S}}
\newcommand{\cT}{{\mathcal T}}

\newcommand{\cX}{{\mathcal X}}
\newcommand{\cY}{{\mathcal Y}}

\newcommand{\as}{\ \text{a.s.}}
\newcommand{\io}{\ \text{i.o.}}

\newcommand{\ubar}[1]{\underline{#1\mkern-4mu}\mkern4mu }
\newcommand{\omu}{{\bar \mu}}
\newcommand{\umu}{{\ubar \mu}}

\newcommand{\cCs}{{\mathcal C}_{s}}

\newcommand{\bigmid}{\; \bigl| \;}
\newcommand{\Bigmid}{\; \Bigl| \;}
\newcommand{\biggmid}{\; \biggl| \;}

\newcommand{\tra}{{\scalebox{0.6}{$\top$}}}

\newlist{myenumi}{description}{10}
\setlist[myenumi]{labelindent=\parindent, leftmargin=*, align=left, itemsep=1pt, parsep=0pt}
\setlist[myenumi]{leftmargin=0pt}

\makeatletter
\def\namedlabel#1#2{\begingroup  
    (#2)%
    \def\@currentlabel{#2}%
    \phantomsection\label{#1}\endgroup
}
\makeatother

\title{Cutpoints of non-homogeneous random walks}
\author{Chak Hei Lo\footnote{Department of Statistical Science, University College London, Gower Street, London WC1E~6BT, UK.} \and Mikhail V.\ Menshikov\footnote{Durham University, Department of Mathematical Sciences, South Road, Durham DH1~3LE, UK.} \and Andrew R.\ Wade\footnotemark[2]} 
\date{\today}

\begin{document}

\maketitle

\begin{abstract}
We give conditions under which near-critical stochastic processes on the half-line
have infinitely many or finitely many cutpoints, generalizing existing results
on nearest-neighbour random walks to adapted processes with bounded increments
satisfying appropriate conditional increment moments conditions. We apply one of these results
to deduce that a class of transient zero-drift Markov chains in $\R^d$, $d \geq 2$,
possess infinitely many separating annuli, generalizing previous results on
spatially homogeneous random walks.
\end{abstract}

\medskip


\medskip


\section{Introduction and main results}
\label{sec:introduction}

In this paper we study separation properties of trajectories of transient, near-critical, 
discrete-time stochastic processes in~$\RP$ and~$\R^d$ satisfying certain increment moment conditions.
A point $x$ of $\RP$   is a \emph{cutpoint} for a given trajectory of a stochastic process if, roughly speaking, 
the process visits $x$ and never returns to $[0,x)$ after its first entry into $(x,\infty)$. A similar notion is applicable in higher dimensions.
Under mild conditions, cutpoints may appear only in the transient case, when trajectories
escape to infinity.
The more cutpoints that a process has, the `more transient' it is, in a certain sense.
 A fundamental question is: does a transient process have  
infinitely many cutpoints, or not? 

For simple symmetric random walk (SSRW) on $\Z^d$, $d \geq 3$, this question goes back to Erd\H os and Taylor~\cite{et},
who proved that cutpoints have a positive density in the trajectory if $d \geq 5$. Much later,
it was shown that transient SSRW has infinitely many cutpoints in any dimension $d \geq 3$, by Lawler~\cite{lawlerbook} (for $d \geq 4$) and James and Peres~\cite{jp}.
Recently, examples of transient Markov chains on $\ZP$ with finitely many cutpoints were produced~\cite{jlp,cfr}:
these processes are nearest-neighbour birth-and-death chains that are `less transient' than SSRW on $\Z^3$,
in the `critical window' identified in~\cite{mai}. A recent extension to processes whose jumps are size 1 to the left
and size 2 to the right can be found in~\cite{wang}.

We examine the phase transition in the quantity of cutpoints
from the point of view of relatively general processes on $\RP$
in the manner of Lamperti~\cite{lamp1,lamp3}, and from the point of view of many-dimensional Markov chains (cf.~\cite{gmmw}).
We  (i) extend, in part, some one-dimensional results that were restricted to nearest-neighbour Markov chains on $\ZP$~\cite{cfr}
to somewhat more general (not always Markov) processes on $\RP$ with bounded jumps satisfying Lamperti-type conditions, and
(ii) extend some many-dimensional results that were restricted to homogeneous random walks in $\R^d$, $d \geq 3$~\cite{jp}
to some transient zero-drift non-homogeneous random walks on $\R^d$, $d \geq 2$. Indeed, it is the fact that we
can undertake relevant parts of~(i) without assuming the Markov property that enables us to apply our results to higher dimensions in~(ii).

Suppose that $X = (X_n; n \in \ZP)$ is a discrete-time stochastic process adapted to a filtration $(\cF_n; n \in \ZP)$ and taking values in a measurable $\cX \subset \RP$ 
with $\inf \cX = 0$ and $\sup \cX = \infty$. We permit $\cF_0$ to be rich enough that $X_0$ is random.
For a measurable subset $B$ of $\RP$, let
$|B|$ denote the Lebesgue measure of $B$.
For a set $A$, let $\# A$ denote the number of elements of~$A$.

\begin{definition} \phantomsection
\label{def:cutpoints}
\begin{myenumi}
\setlength{\itemsep}{0pt plus 1pt}
\item[{\rm (i)}]
The point $x \in \RP$ is a \emph{cutpoint} for $X$ if there exists $n_0 \in \ZP$ such that $X_n \leq x$ for all $n \leq n_0$,
$X_{n_0} = x$, and $X_n > x$ for all $n > n_0$.  
\item[{\rm (ii)}]
The point $x \in \RP$ is a \emph{strong cutpoint} for $X$ if there exists $n_0 \in \ZP$ such that $X_n < x$ for all $n < n_0$,
$X_{n_0} = x$, and $X_n > x$ for all $n > n_0$.  
\item[{\rm (iii)}]
For $h >0$ and $k \in \ZP$,  an interval $I \subset \RP$ is an $(h,k)$ \emph{cut interval} if 
$| I | \geq h$, if there are at least $k$ points of  $X_0, X_1, \ldots$ in the interior of $I$,
and every point of $X_0, X_1, \ldots$ in the interior of $I$ is a strong cutpoint for $X$.
\end{myenumi}
\end{definition}
The terminology in (i) and~(ii) follows~\cite{cfr}, although similar definitions appeared earlier. We discuss some other related notions in Section~\ref{sec:relations} below. 
Let $\cC$ denote the set of cutpoints, and let $\cCs$ denote the set of strong cutpoints;
the random sets $\cC$ and $\cCs$ are at most countable, with $\cCs \subseteq \cC$. 

In this paper we give conditions under which 
either (i) $\# \cCs = \infty$,
or (ii) $\# \cC < \infty$.
The example of a trajectory on $\ZP$ which follows the sequence $(0,0,1,1,2,2,\ldots)$
shows that it is, in principle, possible to have $\# \cC = \infty$ and $\# \cCs < \infty$,
but our results show that such behaviour is excluded for the models that we consider (with probability~1);
see also Conjecture~1.1 of~\cite[p.~628]{cfr}.

We will assume the following.
\begin{description}
\item
[\namedlabel{ass:bounded-jumps}{B}]
Suppose that there exists a constant $B < \infty$ such that, for all $n \in \ZP$,
\[
\Pr ( | X_{n+1} - X_n | \leq B   ) =1 . \]  
\item
[\namedlabel{ass:nonc}{N}]
Suppose that $\limsup_{n \to \infty} X_n = + \infty$, a.s.
\end{description}

Assumption~\eqref{ass:bounded-jumps} is bounded increments.
The non-confinement condition~\eqref{ass:nonc} is implied by suitable notions of irreducibility or
ellipticity (see~e.g.~\cite[\S\S3.3, 3.6]{mpw}); in particular,
condition~\eqref{ass:nonc} holds
whenever~$X$ is an irreducible, time-homogeneous Markov chain
on a locally finite state space $\cX \subseteq \RP$. (A set $\cX \subseteq \RP$ is locally finite
if $\# ( \cX \cap B ) < \infty$ for every bounded $B \subseteq \RP$.)

For $n \in \ZP$ set
$\Delta_n := X_{n+1}-X_n$, the increment of the process.
We will impose conditions on the conditional increment moments $\Exp ( \Delta^k_n \mid \cF_n )$, $k=1,2$, that
are required to hold uniformly (in $n$ and a.s.) on $\{X_n > x\}$ for large enough $x$. Note that the existence of $\Exp ( \Delta^k_n \mid \cF_n )$ for all $n \in \ZP$ is guaranteed by~\eqref{ass:bounded-jumps}. 
These conditions will be formulated in terms of (measurable) functions $\umu_k, \omu_k : \cX \to \R$
such that
\begin{equation}
\label{eq:mu-properties}
\umu_k (X_n) \le \Exp(  \Delta^k_n \mid \cF_n ) \le \omu_k (X_n), \as
\end{equation}
for all $n \in \ZP$. 
Of course, condition~\eqref{ass:bounded-jumps} ensures that such $\umu_k$, $\omu_k$ exist; our results
are stronger the tighter one makes the bounds in~\eqref{eq:mu-properties},
and are more complete  if $\umu_k (x)$ and $\omu_k (x)$ do not differ by much for large $x$. 
One near-optimal way of defining $\umu_k, \omu_k$ 
satisfying~\eqref{eq:mu-properties}
is described in~\cite[\S 3.3]{mpw}.
Note that if $X$ is a time-homogeneous Markov chain on $\cX$, then~$\Exp ( \Delta^k_n \mid \cF_n )
 = \mu_k (X_n)$ a.s.~for some measurable $\mu_k : \cX \to \R$
with $\mu_k(x) =  \Exp ( \Delta^k_n \mid X_n = x )$,
and so in~\eqref{eq:mu-properties} we may take
$\umu_k \equiv \omu_k \equiv \mu_k$. Thus in the Markovian case
one may replace $\umu_k$ and $\omu_k$ by $\mu_k$ in the statements that
follow. 

A mild additional assumption that we will often need is the following.
\begin{description}
\item
[\namedlabel{ass:variance}{V}]
Suppose that $\liminf_{x \to \infty} \umu_2(x)  > 0$.
\end{description}

Our first result gives a sufficient condition to have $\# \cCs = \infty$,
and gives a lower bound on the density of cutpoints. That the hypotheses of Theorem~\ref{thm:infcp}
imply $X_n \to \infty$ a.s.~is a result of Lamperti~\cite{lamp1}.

\begin{theorem}
\label{thm:infcp}
Suppose that~\eqref{ass:bounded-jumps},~\eqref{ass:nonc}, and~\eqref{ass:variance} hold. Suppose also that
\begin{align} 
\label{eq:cond1}
\liminf_{x \to \infty} \bigl( 2x \umu_1(x) -\omu_2(x) \bigr) &>0,  \\
\label{eq:cond2}
\limsup_{x \to \infty} \bigl( x \omu_1(x) \bigr) &< \infty.
\end{align}
Then for any $h \in (0,\infty)$ and $k \in \ZP$, a.s.~there exist infinitely many disjoint $(h,k)$ cut intervals. 
In particular $\Pr ( \#\cCs = \infty ) = 1$. 
Moreover, if $\Exp X_0 < \infty$ then there is a constant $c>0$ such that $\Exp \# ( \cCs \cap [0,x] ) \geq c \log x$
for all $x$ sufficiently large.
\end{theorem}

For the case of a nearest-neighbour Markov chain on $\ZP$, Theorem~\ref{thm:infcp}
is contained in~\cite{cfr}. 
For the model in~\cite{wang}, which lives on $\ZP$ and from $x > 0$ jumps either 1 to the right or 2 to the left,
existence of infinitely many disjoint $(3,2)$ cut intervals (say) implies there
are infinitely many points of $\ZP$ that are never visited by the walk (called `skipped points' in~\cite{wang}). 

To appreciate the context of Theorem~\ref{thm:infcp}, recall Lamperti's result that, under the other conditions
of the theorem, condition~\eqref{eq:cond1}
is sufficient for transience,
while sufficient for recurrence is
$\limsup_{x \to \infty} \bigl( 2x \omu_1(x) -\umu_2(x) \bigr) < 0$: see~\cite{lamp1} or Chapter~3 of~\cite{mpw}.
It is helpful to bear in mind the following example. Suppose that $S_0, S_1, S_2, \ldots$
is SSRW on $\Z^d$. Then if $X_n = \| S_n \|$, the process $X$ satisfies our~\eqref{ass:bounded-jumps} and~\eqref{ass:nonc}, and
a calculation (see e.g.~\cite[\S 1.3]{mpw}) shows that
\[ \lim_{x \to \infty} \bigl( 2 x \umu_1 (x) \bigr) = \frac{d-1}{d} = \lim_{x \to \infty} \bigl( 2 x \omu_1 (x) \bigr) ,\]
and
\[ \lim_{x \to \infty}  \umu_2 (x)   = \frac{1}{d} = \lim_{x \to \infty}   \omu_2 (x) . \]
Thus~\eqref{eq:cond2} and~\eqref{ass:variance} also hold, and condition~\eqref{eq:cond1}
is equivalent to $d-2 > 0$, i.e.~$d \geq 3$. We explore the implications of Theorem~\ref{thm:infcp}
for many-dimensional random walks in more detail in Section~\ref{sec:elliptical}.

To find examples of transient processes with $\# \cC  < \infty$,
we require processes that are `less transient' than SSRW in $\Z^3$.
At this point it is most convenient to assume that $X$ is Markov.
 Then a more refined recurrence classification~(see~\cite{mai})
says that a sufficient condition for transience is, for some $\theta >0$ and all $x$ sufficiently large,
\[ 2 x \mu_1 (x) \geq \left( 1 + \frac{1 + \theta}{\log x } \right) \mu_2 (x) ,\]
and a sufficient condition for recurrence is the reverse inequality with $\theta < 0$.
E.g., 
if  
\begin{equation}
\label{eq:second-criticality}
\lim_{x \to \infty} \mu_2 (x) = b \in (0,\infty), \text{ and }
 \mu_1 (x) = \frac{a}{2x} + \frac{c +o(1)}{2x \log x} ,\end{equation}
then $a > b$ implies Theorem~\ref{thm:infcp} holds, and $a < b$ is recurrent (regardless of $c$).
The critical case has $a=b$, and then $c < b$ implies recurrence and $c > b$ implies transience.
This latter regime provides examples of processes with few cutpoints, as we show in Theorem~\ref{thm:fcp}.

We need for this result some stronger regularity assumptions on the process, as follows.
A set $S \subseteq \RP$ is \emph{uniformly locally finite} if
\begin{equation}
\label{eq:ulf}
\sup_{x \in \RP} \# ( S \cap [x,x+1] ) < \infty . \end{equation}

\begin{description}
\item
[\namedlabel{ass:markov}{M}]
Suppose that $X$ is an irreducible, time-homogeneous Markov chain on
an unbounded, uniformly locally finite state space $\cX \subseteq \RP$. List the elements of $\cX$ in increasing order
as $0 = s_0 < s_1 < s_2 < \cdots$. Suppose that there exist $k_0, m_0 \in \N$ and $\delta_0 >0$ such that
for all $k \geq k_0$ there is $m = m(k)$ with $1 \leq m \leq m_0$ for which
\begin{equation}
\label{eq:one-step}
\Pr \Bigl( X_{m} = s_{k+1}, \, \max_{0 \leq \ell \leq m-1} X_\ell \leq s_k \Bigmid X_0 = s_k \Bigr) > \delta_0.
\end{equation}
\end{description}

Note that~\eqref{ass:markov} implies~\eqref{ass:nonc} (see e.g.~Corollary~2.1.10 of~\cite{mpw}).
Condition~\eqref{eq:one-step} holds with $m_0 = 1$ if $\Pr ( X_{n+1} = s_{k+1} \mid X_n = s_k ) > \delta_0 >0$,
as in the nearest-neighbour model of~\cite{cfr}, 
but also holds for example in the setting of~\cite{wang} (with $m_0=2$).

\begin{theorem}
\label{thm:fcp}
Suppose that~\eqref{ass:markov},~\eqref{ass:bounded-jumps}, and~\eqref{ass:variance} hold. 
Suppose also that there exist constants $x_0 \in \RP$ and $D < \infty$ such that
\begin{equation}
\label{eq:cond3}
\mu_1(x) \ge 0 ~ \text{and} ~
2x\mu_1(x) - \mu_2(x) \le  \frac{D}{\log x} , \text{ for all } x \geq x_0.
\end{equation}
Then $\Pr ( \# \cC < \infty) = 1$.
\end{theorem}

The next result gives a result in the other direction.
In particular, Proposition~\ref{prop:exp-infinite}
gives a mild condition,   not requiring~\eqref{ass:markov},
under which $\Exp  \#\cCs  = \infty$. Note that~\eqref{eq:cond4}
is weaker than~\eqref{eq:cond1} from Theorem~\ref{thm:infcp}.
 Theorem~\ref{thm:fcp} and Proposition~\ref{prop:exp-infinite}
together show that if~\eqref{ass:markov} and~\eqref{eq:second-criticality} hold
with $c > a = b$, then $\#\cC$ and $\# \cCs$ are a.s.~finite (so, in particular, Theorem~\ref{thm:infcp} does not apply), but both have infinite expectation.

\begin{proposition}
\label{prop:exp-infinite}
Suppose that~\eqref{ass:bounded-jumps},~\eqref{ass:nonc}, and~\eqref{ass:variance} hold. Suppose also that
for some $\theta >0$ and all $x$ sufficiently large,
\begin{equation}
\label{eq:cond4}
 2 x \umu_1 (x) \geq \left( 1 + \frac{1 + \theta}{\log x } \right) \omu_2 (x) . \end{equation}
If also $\Exp X_0 < \infty$, 
then there exists a constant $c>0$ such that
$ \Exp (  \#\cCs \cap [0,x] ) \geq c \log \log x$ for all $x$ sufficiently large;
in particular, $\Exp \# \cCs  = \infty$.
\end{proposition}

Proposition~\ref{prop:exp-infinite}
is reminiscent of  Theorem~2 of~\cite{bgs}, which states that for
any transient Markov chain $X$ on a countable state space,
the expected number of \emph{$f$-cutpoints} is infinite, where an $f$-cutpoint is a cutpoint
for the process $f(X_0), f(X_1), \ldots$, and $f$ is a particular function determined by the
law of the Markov chain (see~\cite{bgs}). However, $f$-cutpoints
are not necessarily cutpoints in our sense, since the function $f$ is not necessarily monotone.

\begin{conjecture}
Under the conditions of Theorem~\ref{thm:fcp},  $\Exp [ ( \#\cC )^\alpha ] < \infty$ for all $\alpha < 1$.
\end{conjecture}
	
	We believe that the condition~\eqref{eq:cond3} in	
	Theorem~\ref{thm:fcp} can be relaxed. Indeed,
	in the case of a nearest-neighbour random walk on~$\ZP$,
	it is shown in~Theorem~5.1 of~\cite{cfr} that
	if, for $\gamma >0$,
		\[ 2x\mu_1(x) =  1 + \frac{1}{(\log \log x)^\gamma}, \text{ and } \mu_2 (x) =1, \text{ for all } x \geq 1,\]
then	$\# \cC < \infty$ a.s.~whenever $\gamma >1$, while
$\# \cCs = \infty$ a.s.~if $\gamma \leq 1$. 

\begin{problem}
Obtain a 
sharp phase transition analogous to that in~\cite{cfr} in the generality considered in the present paper.
\end{problem}

\section{Application to higher dimensions}
\label{sec:elliptical}

Suppose that $d \in \N$,
and let $\Sigma$ be an unbounded, measurable subset of $\R^d$ with $0 \in \Sigma$.
Let $\Xi := (\xi_0, \xi_1, \xi_2, \ldots)$ be a time-homogeneous Markov
process with $\Pr ( \xi_n \in \Sigma ) = 1$ for all $n$,
with a family of laws $\Pr_x ( \, \cdot \,) = \Pr ( \, \cdot \, \mid \xi_0 = x)$ for initial
state $x \in \Sigma$.
In other words, for measurable $A \subseteq \Sigma$ and $x \in \Sigma$,
 $\Pr ( \xi_{n+1} \in A \mid \xi_n = x ) = \Pr_x ( \xi_1 \in A ) = P (x, A)$
for a transition kernel $P$.  
Define $\theta_n := \xi_{n+1} - \xi_n$ for $n \in \ZP$,
and write simply $\theta$ for $\theta_0$. Throughout
this section we
view vectors in $\R^d$ as column vectors. Write $\Sp{d-1} := \{ u \in \R^d : \| u \| =1 \}$,
and for $x \in \R^d \setminus \{ 0 \}$, set $\hat x := x / \| x \|$.

Assume that the increments of $\Xi$ are bounded, i.e., for some constant $B < \infty$,
\begin{equation}
\label{eq:xi-bounded}
\Pr_x \left( \| \theta \| \le B   \right) = 1, \text{ for all } x \in \Sigma.
\end{equation}
Under assumption~\eqref{eq:xi-bounded}
the mean drift function $\mu(x) := \Exp_x \theta$ (a vector in $\R^d$) and
increment covariance function $M(x) := \Exp_x ( \theta \theta^\tra )$ (a $d \times d$ symmetric matrix)
are well-defined. Here $\Exp_x$ is expectation with respect to $\Pr_x$.
We assume that the walk has zero drift, i.e.,
\begin{equation}
\label{eq:xi-zero-drift}
\mu (x ) = 0, \text{ for all } x \in \Sigma,
\end{equation}
and is uniformly non-degenerate in the sense that there exists $\eps_0 >0$ such that
\begin{equation}
\label{eq:xi-non-degenerate}
\trace M(x) \geq \eps_0 , \text{ for all } x \in \Sigma.
\end{equation}
A natural class of models consists of the \emph{elliptic random walks} introduced in~\cite{gmmw} (see also~\cite[\S 4.2]{mpw}),
which are described by an asymptotic covariance structure, as follows.
Suppose that there exist constants $U$ and $V$ with $0 < U \leq V < \infty$ for which
\begin{align}
\label{eq:xi-U}
\lim_{r \to \infty} \sup_{x \in \Sigma : \| x \| \geq r} \bigl| \hat x^\tra M (x) \hat x - U \bigr| & = 0, \\
\label{eq:xi-V}
\lim_{r \to \infty} \sup_{x \in \Sigma : \| x \| \geq r} \bigl| \trace M (x)  - V \bigr| & = 0.
\end{align}
Note that if $d=1$, then $M(x)=\Exp_x (\theta^2)$ is a scalar, necessarily $U=V$,
and the process is recurrent (see e.g.~Theorem~2.5.7 of~\cite{mpw}). 
Thus we must take $d \geq 2$ to see transience, and it turns out 
that we must take 
 $2U < V$ (cf.~\cite{gmmw} and~\cite[\S 4.2]{mpw}). The case $2U > V$ is recurrent.
The boundary case $2U=V$ may be recurrent or transient, and, if transient,
may fall into the regime corresponding to Theorem~\ref{thm:fcp}, so we must exclude that case.

\begin{description}
\item
[\namedlabel{ass:xi}{E}]
Suppose that~\eqref{eq:xi-bounded}--\eqref{eq:xi-V} hold with $d \geq 2$ and $2U < V$.
\end{description}

We will obtain results for $\Xi$ by looking at the process $X$
defined by $X_n = \| \xi_n \|$ for $n \in \ZP$, and applying our
one-dimensional results from Section~\ref{sec:introduction}.

For the process $\| \Xi \|$, an $(h,k)$ cut interval
corresponds to an $(h,k)$ \emph{cut annulus} for $\Xi$,
that is, an annulus of width at least~$h$
for which there exist $m$ and $\ell \geq k -1$
with $X_0, \ldots, X_{m-1}$ in the bounded complement of the annulus,
$X_{m+\ell+1}, X_{m+\ell+2}, \ldots$ in the unbounded complement of the annulus,
and $X_m, X_{m+1}, \ldots, X_{m+\ell}$ inside the annulus with
$\| X_m \| < \| X_{m+1} \| < \cdots < \| X_{m+\ell} \|$.

\begin{theorem}
\label{thm:elliptical}
Suppose that $\Xi$ is a time-homogeneous Markov process
on $\Sigma \subseteq \R^d$ 
 for which~\eqref{ass:xi} holds. 
Then a.s., 
for any $h \in (0,\infty)$ and $k \in \N$, there are infinitely many $(h,k)$ cut annuli.
\end{theorem}

The next result on homogeneous random walk, which is essentially due to James and Peres~\cite[\S 4]{jp},
now follows as a special case of Theorem~\ref{thm:elliptical}.

\begin{corollary}
\label{cor:iid}
Suppose that $\zeta, \zeta_1, \zeta_2, \ldots \in \R^d$ are i.i.d.~with
$\Pr ( \| \zeta \| \leq B ) =1$, $\Exp \zeta = 0$, and $\Exp ( \zeta \zeta^\tra ) = \sigma^2 I$,
where $B < \infty$ and $\sigma^2 \in (0,\infty)$ are constants, and $I$ is the $d$ by $d$ identity matrix.
Then for $d \geq 3$, the random walk $\Xi$ generated by $\xi_n = \sum_{i=1}^n \zeta_i$
is transient and, a.s.,  for any $h \in (0,\infty)$ and $k \in \N$, has infinitely many $(h,k)$ cut annuli.
\end{corollary}

Suppose that $\zeta, \zeta_1, \zeta_2, \ldots \in \R^d$ are i.i.d.~with
$\Pr ( \| \zeta \| \leq B ) =1$, $\Exp \zeta = 0$, and $\Exp ( \zeta \zeta^\tra ) = M$
for some positive-definite~$M$.
Then there exists a positive-definite matrix $M^{-1/2}$, which defines a linear transformation of $\R^d$,
such that $M^{-1/2} \Xi$ has increment distribution $\tilde \zeta := M^{-1/2} \zeta$
for which $\Exp \tilde \zeta = 0$ and $\Exp ( \tilde\zeta \tilde\zeta^\tra ) =
M^{-1/2} \Exp  ( \zeta \zeta^\tra ) M^{-1/2} = I$. Thus Corollary~\ref{cor:iid}
applies, showing that, if $d \geq 3$,
 $M^{-1/2} \Xi$  has infinitely many $(h,k)$ cut annuli.
For $\Xi$, this translates to linear transformations (by $M^{1/2}$)
of cut annuli, which are elliptical annuli rather than spherical annuli.
This raises a natural question.

\begin{problem}
For general positive-definite $M$, is it the case that there are infinitely many $(h,k)$ \emph{spherical} cut annuli
for the random walk in $\R^d$, $d \geq 3$, whose increments have mean zero and covariance $M$?
\end{problem}

In looking for transient multidimensional processes with \emph{finitely-many} cut annuli, it is natural to
take processes with a radial drift chosen so that $\| \Xi \|$ has a drift in the window identified by Theorem~\ref{thm:fcp}.
However, Theorem~\ref{thm:fcp} requires the Markov property, and so cannot be applied to $\|\Xi\|$, unless we impose some additional
isotropy condition.

In the rest of this section we give the proofs of Theorem~\ref{thm:elliptical}
and Corollary~\ref{cor:iid}.

\begin{proof}[Proof of Theorem~\ref{thm:elliptical}.]
Let $X_n = \| \xi_n \|$.
Lemma~4.1.1 of~\cite{mpw} shows that~\eqref{eq:xi-bounded}, \eqref{eq:xi-zero-drift}, and~\eqref{eq:xi-non-degenerate}
 imply that $\limsup_{n \to \infty} \| \xi_n \| = \infty$, a.s.
Moreover, Lemma~4.1.5 of~\cite{mpw} shows that under conditions~\eqref{eq:xi-bounded} and~\eqref{eq:xi-zero-drift},
we have that, for some $\delta >0$,
\begin{align*}
\Exp ( \Delta_n \mid \xi_n = x ) & = \frac{\trace M(x) - \hat x^\tra M(x) \hat x}{2 \| x \|} + O ( \| x \|^{-1-\delta} ), \\
\Exp ( \Delta_n^2 \mid \xi_n = x ) & = \hat x^\tra M(x) \hat x + O ( \| x\|^{-\delta} ) ,
\end{align*}
as $\| x\| \to \infty$.
 With~\eqref{eq:xi-U} and~\eqref{eq:xi-V}, we get
\begin{align*}
\Exp ( \Delta_n \mid \xi_n = x )   = \frac{V - U}{2 \| x \|} + o ( \| x \|^{-1} ), \text{ and } \Exp ( \Delta_n^2 \mid \xi_n = x )  = U + o (1) ,
\end{align*}
and so~\eqref{eq:mu-properties} is satisfied with
\[ \umu_1 (x) = \frac{V-U}{2x} + o ( x^{-1} ) ,  ~~~  \omu_1 (x) = \frac{V-U}{2x} + o ( x^{-1} ) ,\]
and $\umu_2 (x) = U + o(1) = \omu_2 (x)$. Since $U >0$, we have that~\eqref{ass:variance} holds,
while
\[ \liminf_{x \to \infty} \bigl( 2 x \umu_1 (x) - \omu_2 (x)  \bigr) =  V - 2 U .\]
Thus if $V > 2U$ we have that~\eqref{eq:cond1} holds.
Then   Theorem~\ref{thm:infcp} gives the result.
\end{proof}

\begin{proof}[Proof of Corollary~\ref{cor:iid}.]
The conditions of Theorem~\ref{thm:elliptical}
are satisfied with $M(x) = \sigma^2 I$, $U = \sigma^2$, and $V = \sigma^2 d$,
so $V > 2U$ if and only if $d > 2$.
\end{proof}

\section{Cut times and separating points}
\label{sec:relations}

A few variations on, and relatives of, the concept of cutpoint have appeared in the literature (see e.g.~\cite{dek,cfr,jlp,jp,lawler96,lawler02}).
Here we briefly comment on a definition of \emph{cut time}, and also introduce the notion of a \emph{separating point}, which will be useful for our proofs.

\begin{definition} \phantomsection
\label{def:cuttimes}
\begin{myenumi}
\setlength{\itemsep}{0pt plus 1pt}
\item[{\rm (i)}]
The point $x \in \RP$ is a \emph{separating point} for $X$ if there exists $n_0 \in \ZP$ such that $X_n \leq x$ for all $n \leq n_0$ and $X_n > x$ for all $n > n_0$.
\item[{\rm (ii)}]
We say that $n \in \ZP$ is a \emph{cut time} for $X$ if $X_n = \max_{0\leq \ell \leq n} X_\ell < X_m$ for all $m > n$.
\end{myenumi}
\end{definition}

Note that, in contrast to a cutpoint, a separating point need not be visited by $X$.
It follows from (ii) that if $n$ is a cut time, then 
$\{ X_0, \ldots, X_n \} \cap \{ X_{n+1}, X_{n+2} ,\ldots \} = \emptyset$, an attribute that
has received some attention in discrete spaces~\cite{lawler96,bgs}, but which does not
ensure the spatial separation properties that interest us here.

Let $\cT \subseteq \ZP$ be the set of  cut times for $X$, and let $\cS \subseteq \RP$ be the set of separating points. 
Note that $\cC \subseteq \cS$, and if $I$ is any open $(h,k)$ cut interval, then $I \subseteq \cS$.
In particular, if for some $h>0$ there are infinitely many disjoint $(h,k)$ cut intervals, then $| \cS | = \infty$.
The next result gives some simple relations involving $\cC$, $\cT$, and $\cS$.

\begin{lemma}
We have (i) $\# \cC = \# \cT$,  and (ii) if $\limsup_{n \to \infty} X_n = \infty$, then $\# \cC  =\infty$ implies $\sup \cS = \infty$.
\end{lemma}
\begin{proof}
If $n$ is a cut time, then $X_n$ is a cutpoint. 
If $n_1 < n_2$ are different cut times, then $X_{n_2} = \max_{0 \leq \ell \leq n_2} X_\ell \geq X_{n_1+1} > X_{n_1}$.
Hence $\# \cC \geq \# \cT$. On the other hand, if $x$ is a cutpoint, then there is $n_0$ for which $\max_{0 \leq n \leq n_0} X_n = X_{n_0} = x$ and $X_n > X_{n_0}$ for all $n > n_0$,
so $n_0$ is a cut time; hence $\# \cT \geq \# \cC$. Thus~(i) holds.

For part~(ii), suppose that $\limsup_{n \to \infty} X_n = \infty$,
and there are $x_1 < x_2 < \cdots \in \cC$
with $X_{n_k} = x_k$ for times $n_1 < n_2 < \cdots$.
Since 
$X_{n_{k+1}} > \max_{0 \leq m \leq n_k} X_m$,
we have $\lim_{k \to \infty} X_{n_k} \geq \limsup_{n \to \infty} X_n = \infty$, 
so $\cC  \subseteq \cS$ is unbounded.
\end{proof}

\section{Hitting probability estimates}
\label{sec:probabilities}

By~\eqref{ass:bounded-jumps} and~\eqref{ass:nonc},
for any $x > X_0$ 
the process $X$ on $\RP$ will visit $[x,x+B]$; consider the first time it does so,
at some $y \in [x,x+B]$, say. We will show (in Lemma~\ref{lem:ellipticity} below)
that, provided $x$ is large enough, there is uniformly positive probability that
on its next few steps the process makes a sequence of uniformly positive increments, to reach $[y+ 2 h,\infty)$, say.
If then $X$ never returns to $[0,y+h]$, the process will only visit $[y,y+h]$ at (strong) cutpoints.
By adjusting constants, we can thus produce an $(h,k)$ cut interval. The key estimate thus required is  the probability that,
started close to, but greater than, $y$, the process never returns to $[0,y]$. We use a Lyapunov function approach to estimate this probability.
Related estimates are required for proving that cutpoints \emph{do not} occur.
Similar hitting probability estimates play a key role in the work of~\cite{jp,jlp,cfr,wang}, which focused on the Markovian case.
We emphasise that the Markov property is not necessary for much of the argument.

We start with the following
elementary lemma, which gives a one-sided `ellipticity' result. Note that
since our increment moment conditions are asymptotic, we get~\eqref{eq:ellipticity} only for $y_0$ sufficiently
large; if we had stronger conditions so that we could take $y_0 = 0$ in~\eqref{eq:ellipticity},
then assumption~\eqref{ass:nonc} would follow automatically (see e.g.~Proposition~3.3.4 of~\cite{mpw}).

\begin{lemma}
\label{lem:ellipticity}
Suppose that~\eqref{ass:bounded-jumps} and~\eqref{ass:variance} hold. Suppose also that
$\liminf_{x \to \infty} \umu_1 (x) \geq 0$. Then there exist $y_0 \in \RP$ and $\eps >0$ such that, for all $n \in \ZP$,
\begin{equation}
\label{eq:ellipticity}
 \Pr ( X_{n+1} - X_n \geq \eps \mid \cF_n ) \geq \eps, \text{ on } \{ X_n \geq y_0 \} .\end{equation}
\end{lemma}
\begin{proof}
Let $B<\infty$ be the constant appearing in~\eqref{ass:bounded-jumps}.
Write $\Delta_n = X_{n+1} -X_n$, and for $z \in \R$ write $z^+ = z \1 { z >0 }$
and $z^- = - z \1 {z < 0}$, so $z = z^+ - z^-$ and $|z| = z^+ + z^-$.
By~\eqref{ass:variance} and~\eqref{eq:mu-properties}, there exist constants $\delta>0$ and $y_1 \in \RP$ such that, for all $n \in \ZP$,
\begin{equation}
\label{eq:variance-lower-bound}
 \Exp ( \Delta_n^2 \mid \cF_n ) \geq \delta, \text{ on } \{ X_n \geq y_1 \} .\end{equation}
By assumption~\eqref{ass:bounded-jumps}, we have
 $\Delta_n^2 \leq B | \Delta_n |$, a.s., so, by~\eqref{eq:variance-lower-bound},
\begin{equation}
\label{eq:delta-sum}
\Exp ( \Delta_n^+ \mid \cF_n) + \Exp ( \Delta_n^- \mid \cF_n ) = 
\Exp ( | \Delta_n | \mid \cF_n ) \geq \frac{\delta}{B}, \text{ on } \{ X_n \geq y_1 \} .\end{equation}
Moreover,   $\liminf_{x \to \infty} \umu_1 (x) \geq 0$ implies that  there exists $y_2 \geq y_1$ such that
\begin{equation}
\label{eq:delta-diff}
 \Exp ( \Delta_n^+ \mid \cF_n ) -  \Exp ( \Delta_n^- \mid \cF_n ) =  \Exp ( \Delta_n \mid \cF_n ) \geq - \frac{\delta}{2B}, \text{ on } \{ X_n \geq y_2 \} .\end{equation}
Then we combine~\eqref{eq:delta-sum} and~\eqref{eq:delta-diff} to get, for all $n \in \ZP$,
\begin{equation}
\label{eq:delta-plus} \Exp ( \Delta_n^+ \mid \cF_n ) \geq \frac{\delta}{4B}, \text{ on } \{ X_n \geq y_2 \} .\end{equation}
Now~\eqref{ass:bounded-jumps} shows that for $\eps_0 >0$, $\Delta_n^+ \leq \eps_0 + B \1 { \Delta_n^+ \geq \eps_0 }$.
Thus from~\eqref{eq:delta-plus} we get
\[ \Pr ( \Delta_n^+ \geq \eps_0 \mid \cF_n ) \geq \frac{1}{B} \left(  \frac{\delta}{4B} - \eps_0 \right), \text{ on } \{ X_n \geq y_2 \} .\]
Choose $\eps_0 = \delta/(8B)$. Then we get
\[ \Pr \Bigl( \Delta_n^+ \geq \frac{\delta}{8B} \Bigmid \cF_n \Bigr) \geq  \frac{\delta}{8B^2}, \text{ on } \{ X_n \geq y_2 \}. \]
 This verifies~\eqref{eq:ellipticity}.
\end{proof}

For the rest of the paper we write $\log^p x := (\log x)^p$. Consider Lyapunov functions $f_{\gamma} : \cX \to (0,\infty)$ and $g_{\nu}: \cX \to (0,\infty)$ defined for $\gamma > 0$ and $\nu > 0$ by
\begin{equation*}
f_{\gamma}(x) :=
\begin{cases}
x^{-\gamma} & \text{if } x \ge 1,\\
1 & \text{if } x < 1.
\end{cases}
\end{equation*}
and 
\begin{equation*}
g_{\nu}(x) :=
\begin{cases}
\log^{-\nu} x  & \text{if } x \ge \re,\\
1 & \text{if } x < \re.
\end{cases}
\end{equation*}
Given a $\sigma$-algebra $\cF$ and $\cF$-measurable random variables $X$ and $Y$, we write $o_X^{\cF}(Y)$ to represent an $\cF$-measurable random variable such that for any $\eps >0$, there exists a finite deterministic constant $x_\eps$ for which $|o_X^{\cF}(Y)| \le \eps Y$ on the event $\{X \ge x_\eps\}$.

The next result, which is central to what follows, provides increment moment estimates for our Lyapunov functions,
and is contained in Lemma~3.4.1 of~\cite{mpw}, incorporating a minor correction to restore the factor of $1/2$ to the $\nu (\nu+1)$ term in~\eqref{eq:g-increment}; the $1/2$ factor arises from the second-order Taylor term
in the last display on p.~104 of~\cite{mpw}, but goes missing by equations~(3.23) and~(3.17) in that reference.

\begin{lemma}
\label{lem:lyapunov-function}
Suppose that~\eqref{ass:bounded-jumps} holds.
 Then, for $\gamma >0$,
\begin{align}
\label{eq:f-increment}
& {} \Exp \left(  f_\gamma ( X_{n+1}) - f_\gamma (X_n) \mid \cF_n \right) \nonumber\\
& {} \qquad {} = -\frac{\gamma}{2} \left[ 2 X_n \Exp ( \Delta_n \mid \cF_n ) - (1 + \gamma) \Exp ( \Delta_n^2 \mid \cF_n ) + o_{X_n}^{\cF_n}(1) \right] X_n^{-\gamma-2} ,
\end{align}
and, for $\nu > 0$,
\begin{align}
\label{eq:g-increment}
&{}  \Exp\left(  g_\nu ( X_{n+1}) - g_\nu (X_n) \mid \cF_n \right) \nonumber \\
&{} \qquad {} = - \frac{\nu}{2} \left[ 2 X_n \Exp ( \Delta_n \mid \cF_n ) - \Exp ( \Delta_n^2 \mid \cF_n )  \right] X_n^{-2} \log^{-\nu-1} X_n \nonumber \\ 
& {} \qquad \quad {} + \frac{1}{2}\nu(\nu+1) \Exp( \Delta_n^2 \mid \cF_n ) X_n^{-2} \log^{-\nu-2} X_n + o_{X_n}^{\cF_n}(X_n^{-2} \log^{-\nu-2} X_n)  .
\end{align}
\end{lemma}

Throughout the paper we define, for $n \in \ZP$ and $x \in \RP$, the stopping times
\begin{equation}
\label{eq:tau-eta}
\tau_{n,x} := \min \{ m \ge n: X_m \le x\}, \text{ and }
\eta_{n,x} := \min \{ m \ge n : X_m >  x\} .\end{equation}
Here and elsewhere we adopt the usual convention that $\min \emptyset := \infty$.
The next two results present our main hitting probability estimates.

\begin{lemma}
\label{lem:proest}
Suppose that~\eqref{ass:bounded-jumps} and~\eqref{ass:variance} hold. Suppose also that the conditions~\eqref{eq:cond1} and~\eqref{eq:cond2} hold. 
For any $\delta, C$ with  $0< \delta < C < \infty$,  there exist constants $x_1, k_1, k_2 \in (0,\infty)$, not depending on~$x$, such that for all $x \ge x_1$
and all $y > 0$,
\begin{align}
\label{eq:calp1}
\frac{k_1}{x} \le \Pr(\tau_{n,x} = \infty \mid \cF_n) &\le \frac{k_2}{x}, \text{ on } \{x+ \delta \le X_n \le x+ C \}, \\
\label{eq:calp2}
\Pr(\eta_{n,x+y} < \tau_{n,x} \mid \cF_n) &\le \frac{k_2 (x+ y)}{xy},  \text{ on } \{X_n \le x+C \}. 
\end{align}
\end{lemma}

\begin{lemma}
\label{lem:proest2}
Suppose that~\eqref{ass:bounded-jumps} and~\eqref{ass:variance} hold. 
\begin{itemize}
\item[(a)]
Suppose that there exists $D < \infty$ such that
\begin{equation}
\label{eq:cond3a}
\umu_1(x) \ge 0 ~ \text{and} ~
2x\omu_1(x) - \umu_2(x) \le  \frac{D}{\log x} ,
\end{equation}
for all $x$ sufficiently large.
For any $C < \infty$,  there exists a constant $k_3 \in (0,\infty)$, not depending on~$x$, such that for all $x \geq 1$,
\begin{align}
\label{eq:calp3} 
\Pr(\tau_{n,x} = \infty \mid \cF_n) \le \frac{k_3}{x \log  x},  \text{ on } \{ X_n \leq x+C \}. 
\end{align}
\item[(b)]
Suppose that~\eqref{eq:cond4} holds.
For any $\delta >0$ there exist constants $x_2, k_4 \in (0,\infty)$, not depending on~$x$, such that for all $x \geq x_2$,
\begin{align}
\label{eq:calp4} 
\Pr(\tau_{n,x} = \infty \mid \cF_n) \ge \frac{k_4}{x \log  x},  \text{ on } \{ X_n \geq x + \delta \}. 
\end{align}
\end{itemize}
\end{lemma}
 
The rest of this section is devoted to the proofs of Lemmas~\ref{lem:proest} and~\ref{lem:proest2}.

\begin{proof}[Proof of Lemma~\ref{lem:proest}.]
The assumption~\eqref{eq:cond1} implies that
there exist $y_1 \geq 1$ and $\eps_0 >0$ such that $2x \umu_1(x) -\omu_2(x) \ge \eps_0$ for all $x \geq y_1$, and hence, for all $m \in \ZP$,
by~\eqref{eq:mu-properties},
\[
2 X_m \Exp ( \Delta_m \mid \cF_m ) - \Exp ( \Delta_m^2 \mid \cF_m ) \ge 2X_m \umu_1(X_m) -\omu_2(X_m) \ge \eps_0, \text{ on } \{ X_m \geq y_1 \} .
\]
Also, by~\eqref{ass:bounded-jumps},
$\Delta_m^2 \leq B^2$, a.s.
Thus by Lemma~\ref{lem:lyapunov-function}, we have that on $\{ X_m \geq y_2 \}$ for $y_2 > y_1$ sufficiently large and all $m \in \ZP$, 
\[
\Exp (  f_\gamma ( X_{m+1}) - f_\gamma (X_m) \mid \cF_m )
 \le -\frac{\gamma}{2} \left( \frac{\eps_0}{2} - \gamma B^2 \right) X_m^{\gamma-2}.
\]
Taking $\gamma >0$ sufficiently small, we then have that 
\begin{equation}
\label{eq:f-supermartingale}
\Exp (  f_\gamma ( X_{m+1}) - f_\gamma (X_m) \mid \cF_m ) \leq 0, \text{ on } \{ X_m \geq y_2 \} .
\end{equation}
Set $Y_m = f_\gamma (X_m)$, fix $n \in \ZP$, and take  $x \geq y_2$.
Then $( Y_{m \wedge \tau_{n,x}} ; m \geq n)$
is a non-negative supermartingale, by~\eqref{eq:f-supermartingale}, and hence  $Y_\infty = \lim_{m \to \infty} Y_{m \wedge \tau_{n,x}}$ a.s.~exists in $\RP$, and
(see e.g.~Theorem~2.3.11 of~\cite{mpw})
$\Exp ( Y_{\tau_{n,x} } \mid \cF_n ) \leq Y_n$, a.s.
It follows that 
\[ f_\gamma (X_n ) \geq \Exp ( Y_{\tau_{n,x}} \1 { \tau_{n,x} < \infty } \mid \cF_n )
\geq f_\gamma (x) \Pr ( \tau_{n,x} < \infty \mid \cF_n ) ,\]
since $f_\gamma$ is non-increasing and $X_{\tau_{n,x}} \le x$ on $\{ \tau_{n,x} < \infty \}$. 
In particular, on $\{X_n \geq x + \delta \}$,
\begin{align*}
 \Pr ( \tau_{n,x} = \infty \mid \cF_n ) & \geq \frac{f_\gamma (x) -f_\gamma (X_n)}{f_\gamma (x) } \\
& \geq 1- \frac{f_\gamma (x+\delta)}{f_\gamma (x) }   = 1 - \left( 1 - \frac{\delta}{x+\delta} \right)^\gamma ,\end{align*}
for $x \geq y_2 \geq 1$.  Hence we get the lower bound in \eqref{eq:calp1}.

For the upper bounds in the lemma, we have from~\eqref{ass:bounded-jumps} and the assumption~\eqref{eq:cond2} that there exists $C < \infty$ such that $2x \omu_1(x) \le C$ for all $x \geq 0$.
  Also, by~\eqref{ass:variance} we know that there exist $y_1 \geq 1$ and $\delta >0$
	such that $\umu_2(x) \ge \delta$ for $x \geq y_1$. Hence by~\eqref{eq:mu-properties},
\[
2 X_m \Exp ( \Delta_m \mid \cF_m ) - (\gamma +1 ) \Exp( \Delta_m^2 \mid \cF_m ) \le C - (\gamma +1 ) \delta, \text{ on } \{ X_m \geq y_1 \}.
\]
Thus by Lemma~\ref{lem:lyapunov-function}, we get, for $y_2 > y_1$ sufficiently large
and all $x \geq y_2$,
\[
\Exp (  f_\gamma (X_{m+1}) - f_\gamma (X_m) \mid \cF_m )
 \ge -\frac{\gamma}{2} \left(C + 1 - (\gamma +1 ) \delta \right) X_m^{\gamma-2}, \text{ on } \{ X_m \geq x \}.
\]
Taking~$\gamma >1$ sufficiently large, we thus obtain, for any $x \geq y_2$,
\begin{equation}
\label{eq:f-submartingale}
\Exp (  f_\gamma ( X_{m+1}) - f_\gamma (X_m) \mid \cF_m ) \geq 0, \text{ on } \{ X_m \geq x \} .
\end{equation}
Fix $n \in \ZP$ and $y>0$, and
set $Y_m = f_\gamma (X_m)$.
The stopping times $\eta_{n,x+y}$ are a.s.~finite, by assumption~\eqref{ass:nonc}.
Then $( Y_{m \wedge \tau_{n,x} \wedge \eta_{n,x+y}} ; m \geq n)$
is a uniformly bounded submartingale, by~\eqref{eq:f-submartingale}, 
with limit $Y_{\tau_{n,x} \wedge \eta_{n,x+y}}$, 
and, by optional stopping
(see e.g.~Theorem~2.3.7 of~\cite{mpw}),
$\Exp ( Y_{\tau_{n,x} \wedge \eta_{n,x+y}} \mid \cF_n ) \geq Y_n$, a.s.
In particular, 
\[  f_\gamma (x + C) \leq f_\gamma ( X_n )   \leq \Exp ( Y_{\tau_{n,x} \wedge \eta_{n,x+y}} \mid \cF_n ), \text{ on } \{ X_n \leq x + C \}. \]
Hence, on $\{ X_n \leq x + C\}$,
\begin{align*}
f_\gamma (x + C)  & \leq \Exp \bigl( f_\gamma (X_{\tau_{n,x}} ) \1 {\tau_{n,x} < \eta_{n,x+y} } \bigmid \cF_n\bigr)
+ \Exp \bigl( f_\gamma (X_{\eta_{n,x+y}} ) \1 {\eta_{n,x+y} < \tau_{n,x}  } \bigmid \cF_n \bigr) \\
& \leq f_\gamma (x - B) \Pr ( \tau_{n,x} < \eta_{n,x+y} \mid \cF_n ) +
f_\gamma (x+y ) \Pr ( \eta_{n,x+y} < \tau_{n,x}  \mid \cF_n ) ,\end{align*}
where we have used the fact that $f_\gamma$ is non-increasing and
 $X_{\tau_{n,x}} \geq x - B$
on $\{ \tau_{n,x} < \infty\}$, by~\eqref{ass:bounded-jumps}.
It follows that, on $\{ X_n \leq x + C \}$,
\begin{equation}
\label{eq:submartingale-prob}  \Pr ( \eta_{n,x+y} < \tau_{n,x} \mid \cF_n ) 
\leq  \frac{f_\gamma (x - B) - f_\gamma ( x + C)}{f_\gamma (x-B) - f_\gamma (x+y)} .\end{equation}
Here, for $x > 1+B$,
\begin{align*}
 \frac{f_\gamma (x - B) - f_\gamma ( x + C)}{f_\gamma (x-B) - f_\gamma (x+y)}  & = \frac{(x-B)^{-\gamma}
- (x+C)^{-\gamma}}{(x-B)^{-\gamma} - (x+y)^{-\gamma}} \\
& = \frac{1 - \left( 1 + \frac{B+C}{x-B} \right)^{-\gamma}}{1 - \left( \frac{x-B}{x+y} \right)^\gamma} \\
& = \left( 1 - \left( 1 + \frac{B+C}{x-B} \right)^{-\gamma} \right)
\left( \frac{x+y}{y+B} \right) \left( \frac{1- \theta}{1 - \theta^\gamma } \right) ,\end{align*}
where $\theta = \frac{x-B}{x+y}$. Since~$\gamma >1$, we have that for all $\theta \in (0,1)$, 
$\frac{1- \theta}{1 - \theta^\gamma } < 1$. Hence, by~\eqref{eq:submartingale-prob},
\[ \Pr ( \eta_{n,x+y} < \tau_{n,x} \mid \cF_n ) 
\leq  \left( 1 - \left( 1 + \frac{B+C}{x-B} \right)^{-\gamma} \right)
\left( \frac{x+y}{y+B} \right) \leq \frac{k_2 (x+y)}{xy} ,\]
for all $x$ sufficiently large and all $y > 0$, 
which gives the upper bound in~\eqref{eq:calp2}.
Moreover, it follows from~\eqref{ass:bounded-jumps} that $\eta_{n,x+y} \to \infty$ a.s.~as $y \to \infty$,
so $\tau_{n,x} = \infty$ if and only if $\tau_{n,x} > \eta_{n,x+y}$ for all $y \in \N$.
Hence,
 since the events $\{\tau_{n,x} > \eta_{n,x+y}\}$ are decreasing in $y$,
\begin{align*}
\Pr(\tau_{n,x} = \infty \mid \cF_n ) &= \Pr\Bigl( \bigcap_{y \in \N} \{ \eta_{n,x+y} < \tau_{n,x} \} \Bigmid \cF_n \Bigr) = \lim_{y \to \infty} \Pr (\eta_{n,x+y} < \tau_{n,x}  \mid \cF_n) 
.\end{align*}
Together with~\eqref{eq:calp2}, this yields the upper bound in \eqref{eq:calp1}.
\end{proof}

\begin{proof}[Proof of Lemma~\ref{lem:proest2}]
For part~(a), 
the idea is similar to the proof of the upper bound in~\eqref{eq:calp1}.
By~\eqref{eq:cond3a} and~\eqref{ass:variance},
 there exist $y_1 \in \RP$ and $\delta >0$ so that,
on $\{X_m > y_1 \}$, 
\begin{align*}
&{} \left( 2 X_m \Exp ( \Delta_m \mid \cF_m ) - \Exp( \Delta_m^2 \mid \cF_m )  \right)\log X_m -  (\nu+1) \Exp ( \Delta_m^2 \mid \cF_m ) \nonumber \\
& \qquad{} \le \left( 2X_m \omu_1(X_m) - \umu_2(X_m)  \right)\log X_m -  (\nu+1) \umu_2(X_m) \nonumber \\
& \qquad{} \le D -  (\nu + 1)\delta . 
\end{align*}
In particular, if we take $\nu >0$ large enough so that
$D -  (\nu +1) \delta < 0$, then we have from Lemma~\ref{lem:lyapunov-function} that, for all $x \geq y_2$ with $y_2$ sufficiently large,
\[ \Exp (  g_\nu ( X_{m+1}) - g_\nu (X_m) \mid \cF_m  ) \ge 0, \text{ on } \{ X_m> x \}. \] 

Fix $n \in \ZP$ and $y>0$, and
set $Y_m = g_\nu (X_m)$.
Then $( Y_{m \wedge \tau_{n,x} \wedge \eta_{n,x+y}} ; m \geq n)$
is a uniformly bounded submartingale, 
with limit $Y_{\tau_{n,x} \wedge \eta_{n,x+y}}$, 
and, by optional stopping
$\Exp ( Y_{\tau_{n,x} \wedge \eta_{n,x+y}} \mid \cF_n ) \geq Y_n$, a.s.
In particular, on $\{ X_n \leq x + C \}$,
\begin{align*}
 g_\nu (x + C) \leq g_\nu ( X_n ) & \leq \Exp ( Y_{\tau_{n,x} \wedge \eta_{n,x+y}} \mid \cF_n ) \\
& \leq g_\nu (x - B) \Pr ( \tau_{n,x} < \eta_{n,x+y} \mid \cF_n ) +
g_\nu (x+y ) \Pr ( \tau_{n,x} > \eta_{n,x+y} \mid \cF_n ) .\end{align*} 
It follows that, on $\{ X_n \leq x + C \}$,
\begin{equation}
\label{eq:submartingale-prob-g}  \Pr ( \eta_{n,x+y} < \tau_{n,x} \mid \cF_n ) 
\leq  \frac{g_\nu (x - B) - g_\nu ( x + C)}{g_\nu (x-B) - g_\nu (x+y )} .\end{equation}
Since $\eta_{n,x+y} \to \infty$ as $y \to \infty$, we get, on $\{X_n \le x+C\}$,  
\begin{align*}
\Pr( \tau_{n,x} = \infty \mid \cF_n )  
& = \lim_{y \to \infty}\Pr ( \eta_{n,x+y} < \tau_{n,x} \mid \cF_n ) \\
& \leq  \frac{g_\nu (x - B) - g_\nu ( x + C)}{g_\nu (x-B)} ,
\end{align*}
by~\eqref{eq:submartingale-prob-g}. It follows that, on $\{X_n \le x+ C\}$, 
\begin{align*}
\Pr(\tau_{n,x} = \infty \mid \cF_n) &\le 1- \frac{\log^{-\nu} (x+C) }{\log^{-\nu} (x-B)}  \le \frac{k_3}{x \log x} ,
\end{align*}
for all $x \geq y_2$ and some positive constant $k_3$. This proves part~(a).

The proof of part~(b) is similar to the proof of the lower bound in~\eqref{eq:calp1},
using the function $g_\nu$ rather than $f_\gamma$, and we omit the details.
\end{proof}

\section{Proof of Theorem~\ref{thm:infcp}: Infinitely many cutpoints}
\label{sec:proof1}

Throughout this section we suppose that the hypotheses of Theorem~\ref{thm:infcp} are satisfied.
Note that~\eqref{eq:cond1} and the fact that $\omu_2 (x) \geq 0$ implies that $\liminf_{x \to \infty} \umu_1 (x) \geq 0$,
so that Lemma~\ref{lem:ellipticity} applies.

Lemma~\ref{lem:proest} shows that the probability of having a cutpoint located around $x$ is about $1/x$.
If the (harder half of the) Borel--Cantelli lemma were applicable, this would suggest that
there are infinitely many cutpoints. However, these events  are not independent across different values of $x$.
To use an appropriate version of the Borel--Cantelli lemma, we will bound the probability
that (roughly speaking) \emph{both} $x$ and $y$ are cutpoints using~\eqref{eq:calp2}. This yields
a \emph{positive} probability that there are infinitely many cutpoints, and then
an appeal to a zero--one law gives the result. This is essentially the same approach as is taken in~\cite[\S 2]{jp} and~\cite[\S 4]{cfr};
our approach makes it clear that the Markov property is not essential. 
We set this up more precisely.

Let $B < \infty$ and $\eps >0$ be the constants appearing in~\eqref{ass:bounded-jumps} and Lemma~\ref{lem:ellipticity}, respectively.
Fix $h > 0$ and $k \in \N$. Choose $\ell \in \N$ such that 
\begin{equation}
\label{eq:ell-bound} \ell\eps   > \max ( h, B k). \end{equation}
Let $n \in \ZP$.
Recall the definitions of $\tau_{n,x}$ and $\eta_{n,x}$ from~\eqref{eq:tau-eta}.
For $x \in \RP$ and $i \in \N$, define the events $E_{n,i,x} := \{ X_{\eta_{n,x} +i} - X_{\eta_{n,x}+i-1} > \eps \}$, and set
\[ A_{n,x} := \left( \bigcap_{i=1}^{2\ell} E_{n,i,x} \right) \cap \{ \tau_{\eta_{n,x}+2\ell,x+\ell\eps} = \infty \} .\]
In words, $A_{n,x}$ occurs if, on its first passage after time $n$ into $[x,\infty)$  the process
takes in succession $2\ell$ positive steps of size at least $\eps$ and subsequently never returns to $[0,x+\ell\eps]$.
If the process never visits $[x,\infty)$ before time $n$ and then $A_x$ occurs, all visits to the interval $I_x := [x,x+ \ell\eps]$
are strong cutpoints. More precisely, we have the following.

\begin{lemma}
\label{lem:Ax}
Suppose that~\eqref{ass:bounded-jumps} holds. Then
for all $x \geq X_0+Bn$, 
$A_{n,x}$ implies that $I_x$ is an $(h,k)$ cut interval.
\end{lemma}
\begin{proof}
On the event $A_{n,x}$, we have that
$X_m < x$ for all $m$ with $n \leq m < \eta_{n,x}$, 
 $x \leq X_{\eta_{n,x}} < X_{\eta_{n,x} +1} < \cdots < X_{\eta_{n,x} + 2 \ell}$,
$X_{\eta_{n,x} + 2 \ell}   >  x + 2 \ell\eps$, and, for all $m \geq \eta_{n,x} + 2\ell$, $X_m > x +\ell\eps$.
Thus $A_{n,x} \cap \{ \max_{0 \leq m \leq n} X_m < x \}$ implies that 
every point of $X$ in the interval $I_x$ is  a strong cutpoint. 
In particular, by~\eqref{ass:bounded-jumps},
for fixed $n \in \ZP$ and all $x > X_0 + Bn$, $A_{n,x}$ implies every point of $X$ in the interval $I_x$ is  a strong cutpoint.

Moreover, if $x \geq X_0 + Bn$,
on $A_{n,x}$ we have from~\eqref{ass:bounded-jumps} that,
for $0 \leq m \leq k-1$, $x < X_{\eta_{n,x}+m} < x + B (m+1) < x + B k < x + \ell\eps$, by~\eqref{eq:ell-bound}.
Thus the event $A_{n,x}$ implies that the interval $I_x$ contains at least $k$ values of $X$, and the interval length is $\ell\eps > h$, by~\eqref{eq:ell-bound}. This gives the result. \end{proof}

Set $q := \max(1, 2 \ell\eps)$. Then for $x, y \in \ZP$ with $x < y$,
 intervals $I_{qx}$ and $I_{qy}$ are disjoint. 
Thus to show that there exist infinitely many  $(h,k)$ cut intervals,
it suffices to show that $A_{n,qx}$ occurs for infinitely many $x \in \ZP$
(we write this event as `$A_{n,qx} \io$'). 
First we show that this event has strictly positive probability, uniformly
over $\cF_n$ for any $n$.

\begin{lemma}
\label{lem:positive-prob}
Under the hypotheses of Theorem~\ref{thm:infcp}, 
for any $\ell \in \N$ satisfying~\eqref{eq:ell-bound}
and any
$q > 0$, 
there is a constant $\delta >0$ such that, for all $n \in \ZP$, 
$\Pr ( A_{n,qx} \io \mid \cF_n ) \geq \delta$, a.s.
\end{lemma}

To prove Lemma~\ref{lem:positive-prob}, we will apply
the following  conditional version of the Kochen--Stone lemma~\cite{ks}.

\begin{lemma}
\label{lem:kochen-stone}
On a probability space $(\Omega,\cF,\Pr)$, let $A_1, A_2, \ldots$
be   events and  $\cG \subseteq \cF$  a $\sigma$-algebra.
Let $a \in \N$ be $\cG$-measurable.
Suppose that $\sum_{m=1}^\infty \Pr (A_m \mid \cG) = \infty$, a.s. Then,
\begin{equation}
\label{eq:kochen-stone}
\Pr(A_m \io \mid \cG) \geq \limsup_{m \to \infty} \frac{ \sum_{i = a}^m \sum_{j=1}^{m-i} \Pr(A_i \mid \cG) \Pr(A_{i+j} \mid \cG) }{\sum_{i=a}^m \sum_{j=1}^{m-i} \Pr(A_i \cap A_{i+j} \mid \cG)}, \as
\end{equation}
\end{lemma}

Several of the standard proofs of the Kochen--Stone lemma~\cite{yan,chandra} admit trivial modifications to yield the conditional result. One route,
following~\cite{yan} in the unconditional case, 
proceeds via a conditional version of the Paley--Zygmund inequality and then a conditional version of the Chung--Erd\H os lemma.
We omit the details.

\begin{proof}[Proof of Lemma~\ref{lem:positive-prob}.]
To apply Lemma~\ref{lem:kochen-stone}, we will obtain a lower bound for $\Pr ( A_{n,x} \mid \cF_n)$ and an upper bound for
$\Pr (A_{n,x} \cap A_{n,x+y} \mid \cF_n )$. 
Fix $n \in \ZP$. Suppose that $x \geq a_n :=   \max (x_1,X_0+2Bn)$, where $x_1$ is the constant in Lemma~\ref{lem:proest}.
Define events
\[ D_x = \{  \tau_{\eta_{n,x}+2\ell,x+\ell\eps} = \infty \}, ~ E_x = \bigcap_{i=1}^{2\ell} E_{n,i,x},  ~\text{and}~
F_{x,y} = \{ \eta_{\eta_{n,x}+2\ell,x+y} < \tau_{\eta_{n,x}+2\ell,x+ \ell \eps} \} ;\]
here we omit the $n$-dependence from the notation to make it less cumbersome, and since we keep $n$ fixed throughout the argument.
Note that $A_{n,x} = D_x \cap E_x$.
Then, since $\eta_{n,x} \geq n$ and $E_x \in \cF_{\eta_{n,x}+2\ell}$,
\[ \Pr(A_{n,x} \mid \cF_n)  =
 \Exp \bigl[ \2{E_x} \Pr (  D_x  \mid \cF_{\eta_{n,x}+2\ell} ) \bigmid \cF_n \bigr] .\]
Provided $x > X_0+Bn$, we have  by~\eqref{ass:bounded-jumps} that $\max_{0 \leq m \leq n} X_m <x$
and hence
 $x \leq X_{\eta_{n,x}} \leq x + B$. Thus 
on the event $E_x$ we have that  $x + 2\ell\eps \leq X_{\eta_{n,x}+ 2\ell} \leq x+ (2\ell+1) B$, so we can apply~\eqref{eq:calp1} in Lemma~\ref{lem:proest} 
to obtain $k_1/ x \leq \Pr (  D_x  \mid \cF_{\eta_{n,x}+2\ell} ) \leq k_2 /x$ on $E_x$, where $k_1, k_2 \in (0,\infty)$
do not depend on~$x > X_0 + Bn$ or on~$n$. Thus, for all $x \geq a_n$,
\[ \frac{k_1}{x} \Pr ( E_x \mid \cF_n ) \leq \Pr(A_{n,x} \mid \cF_n)  \leq \frac{k_2}{x}  .\]
Here  Lemma~\ref{lem:ellipticity} and repeated conditioning shows that
$\Pr ( E_x \mid \cF_n ) \geq \eps^{2\ell}$, on $\{ X_n \geq y_0 \}$.
 We conclude that,
for some constants $c_1, c_2 \in (0,\infty)$, all $n \in \ZP$, and all $x \geq a_n$, 
\begin{equation}
\label{eq:prAx-lower-bound}
\frac{c_1}{x} \leq \Pr(A_{n,x} \mid \cF_n) \leq \frac{c_2}{x}, \text{ on } \{ X_n \geq y_0 \} . \end{equation}

On the other hand, 
by~\eqref{ass:nonc}, $\eta_{\eta_{n,x}+2\ell,x+y} < \infty$, a.s.,
so that, for $x, y \in \RP$ and $n \in \ZP$, 
\[ A_{n,x} \cap A_{n,x+y} = E_x \cap D_{x} \cap E_{x+y} \cap D_{x+y}  = E_x \cap F_{x,y} \cap E_{x+y} \cap D_{x+y}  , \]
up to sets of probability zero.
Suppose that $x \geq a_n$ and $y \geq b:= \lceil (2 \ell + 1) B \rceil$.
Then~\eqref{ass:bounded-jumps} implies that $\eta_{n,x+y} \geq \eta_{n,x} +2\ell$, and so
 also $\eta_{\eta_{n,x}+2\ell,x+y} = \eta_{n,x+y}$, since the walk cannot reach $[x+y, \infty)$
until after time $\eta_{n,x} + 2\ell > n$.
In particular, $F_{x,y} \in \cF_{\eta_{n,x+y}}$, and so
\begin{align*}
& {} \Pr  (A_{n,x} \cap A_{n,x+y} \mid \cF_n  )\\
& {} \qquad {} = \Exp \Bigl[ \Exp \bigl[ \2 {E_x}  \Exp \bigl[ \2{F_{x,y}} \Exp \bigl[ \2 { E_{x+y} } \Pr ( D_{x+y} \mid \cF_{\eta_{n,x+y}+2\ell} )  \bigmid \cF_{\eta_{n,x+y}} \bigr] \bigmid \cF_{\eta_{n,x}+2\ell} \bigr] \bigmid \cF_{\eta_{n,x}} \bigr] \Bigmid \cF_n \Bigr] .
\end{align*} 
It follows from Lemma~\ref{lem:proest} that, on the event $E_{x+y}$, $\Pr ( D_{x+y} \mid \cF_{\eta_{n,x+y}+2\ell} ) \leq \frac{k_2}{x+y}$. Thus
\[  \Pr  (A_{n,x} \cap A_{n,x+y} \mid \cF_n  ) \leq \frac{k_2}{x+y} 
\Exp \Bigl[ \Exp \bigl[ \2 {E_x}  \Pr ( F_{x,y} \mid \cF_{\eta_{n,x}+2\ell} ) \bigmid \cF_{\eta_{n,x}} \bigr] \Bigmid \cF_n \Bigr] .\]
Similarly, on the event $E_x$ we have from Lemma~\ref{lem:proest} that $\Pr ( F_{x,y} \mid \cF_{\eta_{n,x}+2\ell} ) \leq \frac{k_2 (x+y)}{xy}$, so 
\begin{equation}
\label{eq:prAxy-upper-bound} \Pr  (A_{n,x} \cap A_{n,x+y} \mid \cF_n  ) \leq \frac{c_3}{x y} , \end{equation}
for some $c_3 < \infty$, all $n \in \ZP$, and all $x, y$ with $x \geq a_n$  and $y \geq b$.

Consider the numerator in Lemma~\ref{lem:kochen-stone}.
From the lower bound in~\eqref{eq:prAx-lower-bound} we have that for all $n \in \ZP$, all $x \geq a_n$, and all $y \geq 1$,
\[ \Pr (A_{n,qx} \mid \cF_n ) \Pr ( A_{n,q(x+y)} \mid \cF_n ) \geq \frac{c_4}{x(x+y)}, \text{ on } \{ X_n \geq y_0 \} , \]
where $c_4 >0$ depends on~$q$. 
It follows that, on $\{ X_n \geq y_0 \}$,
\begin{align*}
 \sum_{x=a_n}^m \sum_{y=1}^{m-x} \Pr (A_{n,qx} \mid \cF_n ) \Pr ( A_{n,q(x+y)} \mid \cF_n )
& \geq c_4 \sum_{x=a_n}^m \sum_{y=1}^{m-x} \frac{1}{x(x+y)}  \\
& = c_4 \sum_{w=a_n+1}^m \sum_{x=a_n}^{w-1} \frac{1}{xw} \\
& \geq c_4  \sum_{w=a_n+1}^m \frac{\log w}{w} - C_n \sum_{w=1}^m \frac{1}{w} ,\end{align*}
for some $\cF_0$-measurable $C_n < \infty$ depending on $a_n$ but not on~$m$. 
Thus we get
\begin{equation}
\label{eq:kochen-stone-numerator}
  \sum_{x=a_n}^m \sum_{y=1}^{m-x} \Pr (A_{n,qx} \mid \cF_n ) \Pr ( A_{n,q(x+y)} \mid \cF_n )
\geq c_5 \log^2 m - C_n \log m ,\end{equation}
for all $m$ sufficiently large, 
where $c_5 >0$ depends neither on  $n$ nor $m$, and $C_n$ does not depend on $m$.
We turn to the denominator in Lemma~\ref{lem:kochen-stone}.
From~\eqref{eq:prAxy-upper-bound} and
 the upper bound in~\eqref{eq:prAx-lower-bound} we have that, for all $n \in \ZP$,  
\begin{align*}
 \sum_{x=a_n}^m \sum_{y=1}^{m-x} \Pr (A_{n,qx} \cap A_{n,q(x+y)} \mid \cF_n ) 
& \leq \sum_{x=a_n}^{m} \sum_{y=1}^b \Pr (A_{n,qx} \mid \cF_n) \\
& {} \qquad {} + \sum_{x=a_n}^m \sum_{y=b}^{m-x} \Pr (A_{n,qx} \cap A_{n,q(x+y)} \mid \cF_n ) \\
& \leq c_6 \sum_{x=1}^m \frac{1}{x} + c_6 \sum_{x=1}^m \sum_{y=1}^{m} \frac{1}{xy}, \end{align*}
where $c_6 < \infty$ depends on $b$ and~$q$.
It follows that, for all $m \geq 2$, say,
\begin{equation}
\label{eq:kochen-stone-denominator}
\sum_{x=a_n}^m \sum_{y=1}^{m-x} \Pr (A_{n,x} \cap A_{n,x+y} \mid \cF_n )  \leq c_7 \log^2 m ,
\end{equation}
for $c_7 < \infty$ depending neither on $n$ nor $m$.

To finish the proof, we apply Lemma~\ref{lem:kochen-stone}
with the bounds~\eqref{eq:kochen-stone-numerator} and~\eqref{eq:kochen-stone-denominator} to get
\[ \Pr ( A_{n,qx} \io \mid \cF_n) \geq \limsup_{m \to \infty} \frac{ c_5 \log^2 m - C_n \log m }{c_7 \log^2 m} = \frac{c_5}{c_7}  , \as ,\]
where $c_5/c_7$ is non-random, positive, and does not depend on~$n$.
\end{proof}
 
\begin{proof}[Proof of Theorem~\ref{thm:infcp}.]
The proof is finished by an argument similar to~\cite[p.~672]{jp}.
Let $\cF_\infty = \sigma ( \cup_{n \geq 0} \cF_n )$.
Let $I^\infty_{h,k} \in \cF_\infty$ denote the event that there are infinitely many disjoint   $(h,k)$ cut intervals.
As argued earlier (see  Lemma~\ref{lem:Ax} and the subsequent paragraph), if $A_{n,qx}$ occurs for infinitely many~$x$, then $I^\infty_{h,k}$ occurs. Thus
by Lemma~\ref{lem:positive-prob}, for all $n \in \ZP$,
\[ \Pr ( I^\infty_{h,k} \mid \cF_n ) \geq \Pr ( A_{n,qx} \io \mid \cF_n ) \geq \delta, \as \]
Then, by L\'evy's zero--one law (see e.g.~Theorem~5.5.8 of~\cite{durrett}), we have
\begin{equation*}
0 < \delta \leq \lim_{n \to \infty} \Pr ( I^\infty_{h,k} \mid \cF_n ) = \Pr ( I^\infty_{h,k}   \mid \cF_\infty  ) = \2  { I^\infty_{h,k}  }, \as
\end{equation*}
Hence the indicator must be equal to 1, a.s., so $\Pr (  I^\infty_{h,k}  ) = 1$.

Finally, suppose that $\Exp X_0 < \infty$. By Lemma~\ref{lem:Ax}, the expected number of disjoint $(h,k)$ cut intervals in $[0,x]$ is (taking $n=0$ in the definition of $A_{n,qy}$) at least
\begin{equation}
\label{eq:cutpoints-lower-bound}
 \Exp \sum_{y \in \N : X_0 < qy < x} \2 { A_{0,qy}   } \geq \Exp \sum_{y \in \N : 0 < qy < x}  \2 { A_{0,qy}   } -\Exp  X_0, \end{equation}
which, by~\eqref{eq:prAx-lower-bound}, is at least $c \log x$
for some $c>0$ and all $x$ sufficiently large.
\end{proof}

\begin{proof}[Proof of Proposition~\ref{prop:exp-infinite}.]
Similarly to the corresponding part of the proof of Theorem~\ref{thm:infcp},
the expected number of disjoint $(h,k)$ cut intervals in $[0,x]$ is bounded below by~\eqref{eq:cutpoints-lower-bound},
and a similar argument to that for the lower bound in~\eqref{eq:prAx-lower-bound},
but now using~\eqref{eq:calp4}, shows 
\[ \Pr ( A_{n,x}  )  \geq \frac{c_1}{x \log x} ,\]
for some $c_1 >0$ and all $x$ sufficiently large.
\end{proof}

\section{Proof of Theorem~\ref{thm:fcp}: Finitely many cutpoints}
\label{sec:proof2}

To show that there are only finitely many cutpoints, one might initially seek to apply the `easy' half of the Borel--Cantelli lemma.
However, the probability estimates of Lemma~\ref{lem:proest2} give an upper bound
on the probability of finding a cutpoint around $x$ of order $1/(x\log x)$, which is not summable. So some additional work is needed.
The basic idea that we adapt goes back to~\cite[\S 3]{jlp}, and was carried forward by~\cite[\S 3]{cfr}. We explain it now.

Let $x \in \N$.
Define intervals $I_x = [ \frac{x}{2} , 2x ]$ and $J_x = [ x, 2x ]$.
Let $E_x$ denote the event that there is at least one cutpoint in $J_x$:
\[ E_x := \{ \# ( \cC \cap J_x ) \geq 1 \} .\]
Recall the definition of the set $\cS$ of separating points from Definition~\ref{def:cuttimes}, and set
 \begin{equation}
\label{eq:M-F}
 M_x := | \cS \cap I_x | , \text{ and } F_x := \{ \cS \cap [ x-1, x] = \emptyset \} .\end{equation}
We will make use of the simple inequality
\begin{equation}
\label{eq:M-conditional}
 \Exp (M_x \mid \cF_0) \geq \Exp ( M_x \2 {E_x } \mid \cF_0) ,\end{equation}
by obtaining an upper bound for the expectation of $M_x$ and
a lower bound for $M_x$ on the event $E_x$. For the latter,
the idea (following~\cite{jlp}) is that if there is one cutpoint (at $r \in J_x$, say)
then there tend to be many more, since for $y < r$ to be a cutpoint one needs
to visit $r$ before returning to $y$ after the first visit to~$y$. However,
some care is needed in this argument, and it is here that we
need to use the Markov property to ensure that the future and the past are independent.

We will need the following estimate on the probability of
first entering $[x+y,\infty)$ at the point $x+y \in \cX$, before
returning to $[0,x]$, starting from not too close to $x$.

\begin{lemma}
\label{lem:fcp}
Suppose that~\eqref{ass:markov},~\eqref{ass:bounded-jumps}, and~\eqref{ass:variance} hold. 
Suppose also that there exists $x_0 \in \RP$  such that $\mu_1(x) \ge 0$
for all $x \ge x_0$. 
Let $\delta > 0$.
Then there exist $c >0$ and $y_0 \in \RP$ such that,
for all $x,y$ with $x \in \cX$, $x \geq x_0$, $x+y \in \cX$, and $y \geq y_0$,
\[ \Pr ( \eta_{n,x+y} < \tau_{n,x}, \, X_{\eta_{n,x+y}} = x+y \mid \cF_n ) \geq \frac{c}{y} , \text{ on } \{ x+ \delta < X_n \leq x+y \} .\]
\end{lemma}
\begin{proof}
Take $x \geq x_0$, fix $n \in \ZP$, and let $z = y- B$, so $z >0$ whenever $y \geq y_0 > B$. 
Set \[ Y_m =   X_{m \wedge \tau_{n,x} \wedge \eta_{n,x+z} }  \1 { x \leq X_n \leq x+y } , \text{ for } m \geq n. \]
Then, since $\mu_1 (u) \geq 0$ for all $u \geq x$,
 $(Y_m ; m \geq n)$
is a non-negative submartingale, which, by~\eqref{ass:bounded-jumps}, is bounded above by $x+z+B$, 
with $\lim_{m \to \infty} Y_m =  X_{\tau_{n,x} \wedge \eta_{n,x+z} }$ on $\{x \leq X_n \leq x+y\}$.
So, by optional stopping, on $\{ x \leq X_n \leq x+y\}$,
\[ X_n \leq \Exp ( X_{\tau_{n,x} \wedge \eta_{n,x+z}}  \mid \cF_n )
\leq x  \Pr (   \eta_{n,x+z} > \tau_{n,x} \mid \cF_n )  + (x+z+B) \Pr (   \eta_{n,x+z} < \tau_{n,x} \mid \cF_n ) .\]
Thus
\[ \Pr (   \eta_{n,x+z} <  \tau_{n,x} \mid \cF_n ) \geq \frac{\delta}{z+B} , \text{ on } \{ x +\delta < X_n \leq x+y \} .\]
By assumption~\eqref{eq:ulf},
$\# ( \cX \cap [ a, a +B ] ) \leq K < \infty$ for all $a \in \RP$ and some constant $K$. 
Thus
there exists an $\cF_n$-measurable $w \in \cX$ with $x+z \leq w \leq x+z+B$
for which
\begin{equation}
\label{eq:K-bound}
  \Pr (   \eta_{n,x+z} < \tau_{n,x} ,\, X_{\eta_{n,x+z}} = w \mid \cF_n ) 
\geq \frac{\delta}{K (z+B)} , \text{ on } \{ x +\delta < X_n \leq x+y \} .\end{equation}
There are at most $K$ points of $\cX$ in the interval $[w,x+y]$, including $w$ and $x+y$; list them in order as $w =x_0 < x_1 <\cdots < x_k = x+y$, where $k \leq K-1$.
Define $u_j = \sum_{i=0}^{j-1} m (x_i)$ where $m(w) \leq m_0$ is as in~\eqref{ass:markov}.
Then define 
the event
\[ F_{n,x,y} = \left( \bigcap_{j=1}^k \left\{ X_{\eta_{n,x+z} + u_j} =x_j \right\} \right) \cap \left\{ \max_{0 \leq \ell < u_k} X_{\eta_{n,x+z}+\ell} < x+y \right\} .\]
By application of~\eqref{eq:one-step}, we see that, on $\{ X_{\eta_{n,x+z}} = w \}$,
\begin{equation}
\label{eq:k-steps}
 \Pr ( F_{n,x,y} \mid \cF_{\eta_{n,x+z}} ) \geq \delta_0^{K} ,\end{equation}
uniformly in $w$,
for $y \geq y_0$ sufficiently large. If  $F_{n,x,y}$ occurs, then after time $\eta_{n,x+z}$,
the process (i) visits $x+y$ without entry into $(x+y,\infty)$, and (ii) by~\eqref{ass:bounded-jumps},
does not return to $[0, x+ z - B K m_0]$ before it reaches $x+y$.
In particular, taking $z \geq B K m_0$ and combining~\eqref{eq:K-bound} and~\eqref{eq:k-steps}, we get
\[  \Pr ( \eta_{n,x+y} < \tau_{n,x}, \, X_{\eta_{n,x+y}} = x+y \mid \cF_n ) \geq \delta_0^{K} \frac{\delta}{K y} , \text{ on } \{ x +\delta < X_n \leq x+y \} ,\]
which completes the proof.
\end{proof}

\begin{proof}[Proof of Theorem~\ref{thm:fcp}.]
We first get an upper bound for the left-hand side of~\eqref{eq:M-conditional}.
Recall that $F_x$ as defined at~\eqref{eq:M-F} is the event that $[x-1,x] \cap \cS = \emptyset$.
For $x \in \RP$, let $\eta_x := \min \{ n \in \ZP : X_n \geq x \}$.
If $\tau_{\eta_{x},x-1} < \infty$, then $X$
returns to $[0,x-1]$ after entering $[x,\infty)$, which implies $F_x$.
Since, by~\eqref{ass:bounded-jumps}, $x \leq X_{\eta_{x}} \leq x +B$ for all $x > X_0$, we may apply Lemma~\ref{lem:proest2}(a) at the stopping time
$\eta_{n,x}$ to obtain
\[ \Pr \bigl( F^\rc_x \bigmid \cF_{\eta_{n,x}} \bigr) \leq \Pr \bigl( \tau_{\eta_{n,x},x-1} = \infty \bigmid \cF_{\eta_{n,x}} \bigr)
\leq \frac{C}{x \log x} ,\]
for some constant $C<\infty$ and all $x > X_0$.
Thus there exists a constant $C<\infty$ for which
\begin{equation}
\label{eq:M-bound1}
 \Exp ( M_x \mid \cF_0 ) \leq   \sum_{y \in \N : [y-1,y ] \cap I_x \neq \emptyset} \Pr ( F^\rc_y \mid \cF_0 )
\leq \frac{C}{\log x} ,\end{equation}
for all $x > 2X_0$.

Next we establish a lower bound for the right-hand side of~\eqref{eq:M-conditional}.
If $E_x$ occurs, set $R_x := \sup ( \cC \cap J_x )$.
Since $\cC \subseteq \cX$ is locally finite, the set $\cC \cap J_x$
is finite and so $R_x \in \cC$ is the rightmost cutpoint in $J_x$.
If $E_x$ does not occur, set $R_x = \infty$.
Then we can write
\[ \Exp ( M_x \2 {E_x} ) = \sum_{r \in \cX \cap J_x} \Exp ( M_x \1 {R_x = r} ) .\]
If $X_{\eta_r} > r$ and $R_x = r$,
 then $r \in \cC$ and so $X_n = r$ for some $n > \eta_r$. But this contradicts the fact that
$r \in \cC$. Thus we have established that 
\begin{equation}
\label{eq:hits-r}
 \{ R_x = r \} \subseteq \{ X_{\eta_r} = r \} ,\end{equation}
up to events of probability zero. 

Let $\eps>0$ be the constant in Lemma~\ref{lem:ellipticity},
and choose $\ell \in \N$ with $\ell \eps > 1$.
For $r \in \cX$ and $y >0$ with $y + 2 \ell B < r$,
let $F_{y,r}$ denote the event
\[ F_{y,r}  = \left( \bigcap_{m = \eta_y}^{\eta_y + \ell} \{ \Delta_m > \eps \} \right)
\cap \left\{ \eta_{\eta_y+\ell+1,r} < \tau_{\eta_y+\ell+1,y+1} \right\} .\]
If $F_{y,r}$ occurs, then on the first visit to $[y,\infty)$,
the process  proceeds via positive steps to $[y+1,\infty)$ and then
visits~$[r,\infty)$ before returning to $[0,y+1]$.
In particular, $F_{y,r} \cap \{ R_x = r \}$ implies that
$(y,y+1) \subseteq \cS$. 
Let $\cY_{x,r} = \{ y_1, \ldots, y_k \}$ be a subset of $\cX$ contained in $I_x \cap [ 0, r - 2 \ell B]$
such that $y_1 < \cdots < y_k$ satisfy $y_1 < \frac{x}{2} +B$, $y_k > r - 2\ell B - B$,
and $1 < y_i - y_{i-1} \leq 2+B$ for all $2 \leq i \leq k$. Existence
of suitable $y_i$ is assured by~\eqref{ass:bounded-jumps} and~\eqref{ass:markov}: $y_1$ can be the
 first point of $\cX$ to the right of $x/2$, and given $y_{i}$, $i \geq 1$, we can take for $y_{i+1}$ the first point of $\cX$ to the right of $y_{i}$
at distance greater that 1. Note that $k$ is bounded below by a constant times $x$ for all $x$ sufficiently large.
Then intervals $[y,y+1]$ are disjoint for different $y \in \cY_{x,r}$, and so
\begin{align}
\label{eq:M-lower1}
  \Exp ( M_x \1 {R_x = r} \mid \cF_0 )
& \geq \Exp \biggl[ \sum_{y \in \cY_{x,r},\, y > X_0 } \Exp ( \2 { F_{y,r}} \1 {R_x = r} \mid \cF_{\eta_r} ) \biggmid \cF_0 \biggr]  \nonumber\\
& \geq \Exp \biggl[ \sum_{y \in \cY_{x,r},\, y > X_0 }   \2 { F_{y,r}}  \Pr ( R_x = r \mid \cF_{\eta_r} ) \biggmid \cF_0 \biggr] ,\end{align}
since~\eqref{ass:bounded-jumps}
means that, provided $y > X_0$ and $y + 2 \ell B < r$, $X_{\eta_y} \leq y + B$ and $\eta_r = \eta_{\eta_y+\ell+1,r} \geq \eta_y +\ell+1$, so that
 $F_{y,r} \in \cF_{\eta_r}$.
Now 
the strong Markov property implies that
$\Pr ( R_x = r \mid \cF_{\eta_r} ) = h ( X_{\eta_r} )$, a.s.,
for some measurable function~$h$ with $\Pr ( R_x = r \mid X_{\eta_r} = z ) = h(z)$.
But~\eqref{eq:hits-r} shows that $h ( z ) = 0$ unless $z = r$,
so $\Pr ( R_x = r \mid \cF_{\eta_r} ) = h ( r ) \1 { X_{\eta_r} = r }$.
Thus from~\eqref{eq:M-lower1} we get
\begin{align*}
 \Exp ( M_x \1 {R_x = r} \mid \cF_0 )
& \geq h (r ) \sum_{y \in \cY_{x,r}, \, y > X_0 }   \Pr (  F_{y,r} \cap \{ X_{\eta_r} = r \} \mid \cF_0 )  \\
& \geq  c h(r)  \sum_{y \in \cY_{x,r}, \, y > X_0  } \frac{1}{r-y} ,\end{align*}
by Lemmas~\ref{lem:ellipticity} and~\ref{lem:fcp}, where $c>0$ is a constant, and $x \geq x_0$.
If $x > 2X_0$, then set $\cY_{x,r}$, taken in reverse order, consists of order $x$ points all of comparable spacing started a constant distance from~$r$,
so we get
$\Exp ( M_x \1 {R_x = r} \mid \cF_0 ) \geq c h(r) \log x$
for all $x > \max(x_0,2X_0)$,
where $c>0$ is again a positive constant. It follows 
that
\[ \Exp ( M_x \2 {E_x} \mid \cF_0) \geq c \sum_{r \in \cX \cap J_x} h(r) \log x, \text{ for all } x > \max(x_0,2X_0) .\]
On the other hand, by a similar argument,
\[ \Pr ( E_x \mid \cF_0 ) = \sum_{r \in \cX \cap J_x} \Pr ( {R_x = r} \mid \cF_0 ) \leq \sum_{r \in \cX \cap J_x} h(r) ,\]
so that
\begin{equation}
\label{eq:M-bound2}
 \Exp ( M_x \2 {E_x} \mid \cF_0)  \geq c \Pr ( E_x \mid \cF_0 ) \log x, \text{ for all } x > \max(x_0,2X_0) .\end{equation}
Combining~\eqref{eq:M-bound1} and~\eqref{eq:M-bound2}, we obtain from~\eqref{eq:M-conditional} that
\[ \Pr ( E_x \mid \cF_0) \leq \frac{C}{\log^2 x } ,\]
for some $C < \infty$ and all $x > \max(x_0,2X_0)$.
Applied along the sequence $x= 2^k$, $k \in \N$,
the (conditional) Borel--Cantelli lemma then shows that
$E_{2^k}$ occurs for only finitely many~$k$, a.s.
The sets $J_{2^k}$, $k \in \N$, cover $[1,\infty)$ and thus $\# \cC < \infty$, a.s.
\end{proof}

\section*{Acknowledgements}

The main part of this work was done while CHL was affiliated to the School of Mathematics, 
University of Edinburgh. The authors gratefully acknowledge an anonymous referee, whose careful reading of the paper and suggestions prompted several corrections and clarifications.

\end{document}